\documentclass[a4,12pt]{amsart}
\usepackage{amssymb,amstext,amsmath,amscd,amsthm,amsfonts,enumerate,latexsym}
\usepackage{color}
\usepackage[dvipdfmx]{graphicx}
\usepackage[all]{xy}

\theoremstyle{plain}
\newtheorem{thm}{Theorem}[section]

\newtheorem*{thm*}{Theorem}
\newtheorem*{cor*}{Corollary}

\newtheorem{prop}[thm]{Proposition}

\newtheorem{lem}[thm]{Lemma}
\newtheorem{cor}[thm]{Corollary}

\newtheorem*{claim*}{Claim}

\theoremstyle{definition}
\newtheorem{defn}[thm]{Definition}

\newtheorem{ex}[thm]{Example}

\newtheorem{fact}[thm]{Fact}

\newtheorem{setting}[thm]{Setting}

\theoremstyle{remark}

\numberwithin{equation}{thm}

\def\Ext{\operatorname{Ext}}

\def\Im{\operatorname{Im}}

\def\Hom{\operatorname{Hom}}

\def\rank{\mathrm{rank}}

\def\e{\mathrm{e}}
\def\m{\mathfrak m}
\def\n{\mathfrak n}

\def\K{\mathrm{K}}

\newcommand{\rme}{\mathrm{e}}

\newcommand{\rmr}{\mathrm{r}}

\newcommand{\rmK}{\mathrm{K}}

\newcommand{\rmQ}{\mathrm{Q}}

\newcommand{\calR}{\mathcal{R}}
\newcommand{\calS}{\mathcal{S}}
\newcommand{\calT}{\mathcal{T}}

\newcommand{\fka}{\mathfrak{a}}

\newcommand{\fkc}{\mathfrak{c}}

\newcommand{\fkm}{\mathfrak{m}}

\newcommand{\fkp}{\mathfrak{p}}

\newcommand{\mapright}[1]{%
\smash{\mathop{%
\hbox to 1cm{\rightarrowfill}}\limits^{#1}}}

\newcommand{\mapleft}[1]{%
\smash{\mathop{%
\hbox to 1cm{\leftarrowfill}}\limits_{#1}}}

\def\depth{\operatorname{depth}}

\def\tr{\mathrm{tr}}

\title[AGL rings arising from fiber products]{Almost Gorenstein rings arising from fiber products}

\author{Naoki Endo}
\address{Global Education Center, Waseda University, 1-6-1 Nishi-Waseda, Shinjuku-ku, Tokyo 169-8050, Japan}
\email{naoki.taniguchi@aoni.waseda.jp}
\urladdr{http://www.aoni.waseda.jp/naoki.taniguchi/}

\author{Shiro Goto}
\address{Department of Mathematics, School of Science and Technology, Meiji University, 1-1-1 Higashi-mita, Tama-ku, Kawasaki 214-8571, Japan}
\email{shirogoto@gmail.com}

\author{Ryotaro Isobe}
\address{Department of Mathematics and Informatics, Graduate School of Science and Technology, Chiba University, 1-33 Yayoi-cho, Inage-ku, Chiba, 263-8522, Japan}
\email{r.isobe.math@gmail.com}


\thanks{ The first author was partially supported by JSPS Grant-in-Aid for Young Scientists (B) 17K14176 and Waseda University Grant for Special Research Projects 2019C-444, 2019E-110. The second author was partially supported by JSPS Grant-in-Aid for Scientific Research (C) 16K05112.}

\begin{document}
\maketitle

\setlength{\baselineskip}{14.7pt}

\begin{abstract}
The purpose of this paper is, as part of the stratification of Cohen-Macaulay rings, to investigate the question of when the fiber products are almost Gorenstein rings. We show that the fiber product $R \times_T S$ of Cohen-Macaulay local rings $R$, $S$ of the same dimension $d>0$ over a regular local ring $T$ with $\dim T=d-1$ is an almost Gorenstein ring if and only if so are $R$ and $S$. Besides, the other generalizations of Gorenstein properties are also explored.
\end{abstract}



\section{Introduction}

The fiber product of homomorphisms $R \overset{f}{\longrightarrow} T \overset{g}{\longleftarrow} S$ of rings is defined by  
$$
R\times_TS = \{(a, b) \in R \times S \mid f(a) = g(b)\}
$$
which forms a subring of $R \times S$. The notion of fiber products, in general, appears not only commutative algebra but in diverse branches of mathematics, and it often plays a crucial role in many situations. There are a huge number of preceding researches about a relation of ring theoretic properties (e.g., Cohen-Macaulay, Gorenstein) between the fiber products and the base rings. The reader may consult \cite{AAM, CSV, Marco, Mo, NS, NSTV, Ogoma, Shapiro} for details. Among them, by \cite{Ogoma}, the fiber product of one-dimensional Cohen-Macaulay local rings over the residue class field is Gorenstein if and only if the base rings are discrete valuation rings (abbr. DVRs), and hence the fiber products are rarely Gorenstein rings. Nevertheless, for example, the fiber product of hypersurfaces might have good properties even though it is not Gorenstein, which we will clarify in the present paper.

An almost Gorenstein ring is one of the generalizations of a Gorenstein ring, defined by a certain embedding of the rings into their canonical modules. The motivation of the theory derives from the strong desire to stratify Cohen-Macaulay rings, finding new and interesting classes which naturally cover the Gorenstein rings. The class of almost Gorenstein rings could be a very nice candidate for such classes. Originally, the theory of almost Gorenstein rings was established by V. Barucci and R. Fr{\"o}berg \cite{BF} in the case where the local rings are analytically unramified of dimension one. They mainly considered the numerical semigroup rings and gave pioneering results. However, since the notion given by \cite{BF} was not flexible for the analysis of analytically ramified case, so that in 2013 the second author, N. Matsuoka and T. T. Phuong \cite{GMP} proposed the notion over one-dimensional Cohen-Macaulay local rings, using the first Hilbert coefficient of canonical ideals. Finally, in 2015 the first and second authors, and R. Takahashi \cite{GTT} gave the definition of almost Gorenstein graded/local rings of arbitrary dimension, by means of Ulrich modules.

The purpose of this paper is to explore the question of when the fiber product is an almost Gorenstein ring, and the main result of this paper is stated as follows.

\begin{thm}\label{1.1}
Let $(R, \m), (S, \n)$ be Cohen-Macaulay local rings with $d=\dim R=\dim S>0$, and $T$ be a regular local ring with $\dim T=d-1$ possessing an infinite residue class field. Let $f: R \to T$, $g: S \to T$ be surjective homomorphisms. Suppose that $A=R \times_T S$ has the canonical module $\rmK_A$ and that $\rmQ(A)$ is a Gorenstein ring. Then the following conditions are equivalent. 
\begin{enumerate}[$(1)$]
\item The fiber product $A=R\times_TS$ is an almost Gorenstein ring.
\item $R$ and $S$ are almost Gorenstein rings.
\end{enumerate}
\end{thm}


In what follows, unless otherwise specified, let $R$ be a one-dimensional Cohen-Macaulay local ring with maximal ideal $\fkm$. Let $\rmQ(R)$ be the total quotient ring of $R$ and $\overline{R}$ the integral closure of $R$ in $\rmQ(R)$. For each finitely generated $R$-module $M$, let $\mu_R(M)$ (resp. $\ell_R(M)$) denote the number of elements in a minimal system of generators of $M$ (resp. the length of $M$). For $R$-submodules $X$, $Y$ of $\rmQ(R)$, let $X:Y = \{a \in \rmQ(R) \mid a Y \subseteq X\}$. We denote by $\mathrm{K}_R$ the canonical module of $R$. We set $\rmr(R) = \ell_R(\Ext^1_R(R/\m, R))$ the Cohen-Macaulay type of $R$.


\section{Basic facts}

The purpose of this section is to summarize some basic properties of fiber products. 
For a moment, let $R$, $S$, $T$ be arbitrary commutative rings, and $f:R \to T$, $g: S \to T$ denote the homomorphism of rings. We set 
$$
A = R\times_TS = \{(a, b) \in R \times S \mid f(a) = g(b)\}
$$
and call it {\it the fiber product of $R$ and $S$ over $T$ with respect to $f$ and $g$}, which forms a subring of $B=R \times S$. We then have a commutative diagram
\[
\xymatrix{
          &
R \ar[rd]^{f}  &  \\
A \ar[rd]_{p_2}\ar[ur]^{p_1} &  & T\\
& S \ar[ur]_{g}
}
\]
of rings, where $p_1 : A \to R, (x, y)\mapsto x$, $p_2: A \to S, (x, y)\mapsto y$ stand for the projections. Hence we get an exact sequence
$$
0 \longrightarrow A \overset{i}{\longrightarrow} B \overset{\varphi}{\longrightarrow} T
$$
of $A$-modules, where $\varphi = \left[\begin{smallmatrix}
f \\
-g
\end{smallmatrix}\right]$. The map $\varphi$ is surjective if either $f$ or $g$ is surjective. Besides, if both $f$ and $g$ are surjective, then $T$ is cyclic as an $A$-module, so that $B$ is a module-finite extension over $A$. 
Therefore we have the following (see e.g., \cite{AAM}).

\begin{lem}\label{2.1}
Suppose that $f:R \to T$ and $g: S \to T$ are surjective. Then the following assertions hold true. 
\begin{enumerate}[$(1)$]
\item $A$ is a Noetherian ring if and only if $R$ and $S$ are Noetherian rings.
\item $(A, J)$ is a local ring if and only if  $(R, \m)$ and $(S, \n)$ are local rings. When this is the case, $J = (\m \times \n) \cap A$.
\item If $(R, \m), (S, \n)$ are Cohen-Macaulay local rings with $\dim R=\dim S = d>0$ and $\depth T\ge d-1$, then $(A, J)$ is a Cohen-Macaulay local ring and $\dim A = d$.
\end{enumerate}
\end{lem}

Suppose that $(R, \m)$, $(S, \n)$ are Noetherian local rings with a common residue class field $k=R/\m=S/\n$, and $f : R \to k$, $g : S \to k$ denote the canonical surjective maps. 

With this notation, we then have the following. Although the assertion $(3)$ follows from the result of Ogoma (\cite{Ogoma}), let us give a brief proof for the sake of completeness. For a Noetherian local ring $R$, we denote by $\rme(R)$ (resp. $v(R)$) the multiplicity (resp. the embedding dimension) of $R$. 

\begin{prop}\label{2.2}
 The following assertions hold true.
\begin{enumerate}[$(1)$]
\item $v(A) = v(R) + v(S)$.
\item If $\dim R=\dim S >0$, then $\rme(A) = \rme(R) + \rme(S)$.
\item If $R$, $S$ are Cohen-Macaulay local rings with $\dim R=\dim S=1$, then $A=R\times_k S$ is Gorenstein if and only if  $R$ and $S$ are DVRs.
\end{enumerate}
\end{prop}

\begin{proof}
Since $J=\m\times \n$ is the maximal ideal of $A$, we get $J^{\ell} = \m^{\ell + 1}\times \n^{\ell + 1}$ for every $\ell \ge 0$. 

$(1)$ Follow from the equalities $\ell_A(J/J^2) = \ell_k([\m/\m^2] \oplus [\n/\n^2]) =\ell_R(\m/\m^2) + \ell_S(\n/\n^2)$.

$(2)$ Remember that $f$ and $g$ are surjective. We then have the equalities
\begin{eqnarray*}
\ell_A(A/J^{\ell+1}) &=& \ell_A(A/J) + \ell_A(J/J^{\ell +1}) \\
&=& 1 + \left[\ell_R(\m/\m^{\ell + 1}) + \ell_S(\n/\n^{\ell + 1})\right] \\
&=& 1+\left\{[\ell_R(R/\m^{\ell + 1})-1] + [\ell_S(S/\n^{\ell + 1})-1]\right\}\\
&=& \left[\ell_R(R/\m^{\ell +1})+ \ell_S(S/\n^{\ell +1})\right] -1
\end{eqnarray*}
for every $\ell \ge 0$, which yield that $\rme(A) = \rme(R) + \rme(S)$, because $\dim A = \dim R = \dim S > 0$.

$(3)$ Suppose that $A$ is a Gorenstein ring. By applying the functor $\Hom_A(-, A)$ to the exact sequence $0 \to A \overset{i}{\to} B \overset{\varphi}{\to} k \to 0$ of $A$-modules (here $\varphi = \left[\begin{smallmatrix}
f \\
-g
\end{smallmatrix}\right]$), we get 
$$
0 \to A:B \overset{i}{\to} A \to \Ext^1_A(A/J, A) \to 0
$$
which implies $J = A:B$, because $\Ext^1_A(A/J, A) \cong A/J$. As $A$ is Gorenstein, we have $B= A:(A:B)$, so that the equalities
$$
R \times S = B= A:(A:B) = A:J = J:J = (\m : \m) \times (\n : \n)
$$
hold where the fourth equality comes from the fact that $A$ is not a DVR (Note that $B$ is a module-finite birational extension of $A$). Consequently, $R=\m:\m$ and $S=\n:\n$, whence $R$ and $S$ are DVRs (see e.g., \cite[Lemma 3.15]{GMP}). The converse implication follows from the assertion $(2)$. This completes the proof.
\end{proof}


\section{Survey on almost Gorenstein rings}

This section is devoted to the definition and some basic properties of almost Gorenstein rings, which we will use throughout this paper. Let $(R, \m)$ be a Cohen-Macaulay local ring with $d=\dim R$, possessing the canonical module $\rmK_R$. 

\begin{defn}\label{3.1}({\cite[Definition 1.1]{GTT}})
We say that $R$ is {\it an almost Gorenstein local ring}, if there exists an exact sequence
$$
0 \to R \to \rmK_R \to C \to 0
$$
of $R$-modules such that $\mu_R(C) = \rme^0_{\m}(C)$, where $\mu_R(C)$ denotes the number of elements in a minimal system of generators for $C$ and
$$
\rme^0_{\m}(C) = \lim_{n \to \infty}\frac{\ell_R(C/\m^{n+1}C)}{n^{d-1}}\cdot (d-1)!
$$ 
is the multiplicity of $C$ with respect to $\fkm$.
\end{defn}

Notice that every Gorenstein ring is an almost Gorenstein ring, and the converse holds if the ring $R$ is Artinian (\cite[Lemma 3.1 (3)]{GTT}). Definition \ref{3.1} requires  that if $R$ is an almost Gorenstein ring, then $R$ might not be Gorenstein but the ring $R$ can be embedded into its canonical module $\K_R$ so that the difference $\K_R/R$ should have good properties.
For any exact sequence
$$
0 \to R \to \rmK_R \to C \to 0
$$
of $R$-modules, $C$ is a Cohen-Macaulay $R$-module with $\dim_RC=d-1$, provided $C \ne (0)$ (\cite[Lemma 3.1 (2)]{GTT}).
Suppose that $R$ possesses an infinite residue class field $R/ \m$. Set $R_1=R/[(0):_RC]$ and let $\m_1$ denote the maximal ideal of $R_1$. 
Choose elements $f_1, f_2, \ldots, f_{d-1} \in \m$ such that $(f_1, f_2, \ldots, f_{d-1})R_1$ forms a minimal reduction of $\m_1$. We then have
$$
\rme_{\m}^0(C) = \rme_{\m_1}^0(C) = \ell_R(C/(f_1, f_2, \ldots, f_{d-1})C) \ge \ell_R(C/\m C) = \mu_R(C). 
$$
Therefore, $\rme_{\m}^0(C) \ge \mu_R(C)$ and we say that $C$ is {\it an Ulrich $R$-module} if $\rme_{\m}^0(C) = \mu_R(C)$, since $C$ is {\it a maximally generated maximal Cohen-Macaulay $R_1$-module} in the sense of \cite{BHU}. Hence, $C$ is an Ulrich $R$-module if and only if $\m C = (f_1, f_2, \ldots, f_{d-1})C$. If $d=1$, then the Ulrich property for $C$ is equivalent to saying that $C$ is a vector space over $R/\m$. Besides, we have the following.

\begin{fact}[\cite{GMP, GTT, K}]\label{3.2}
Suppose that $d=1$ and there exists an $R$-submodule $K$ of the total ring $\rmQ(R)$ of fractions such that $R \subseteq K \subseteq \overline{R}$ and $K \cong \rmK_R$ as an $R$-module, where $\overline{R}$ stands for the integral closure of $R$ in $\rmQ(R)$. Then the following conditions are equivalent.
\begin{enumerate}[$(1)$]
\item $R$ is an almost Gorenstein ring.
\item $\m K \subseteq R$, that is $\m K = \m$.
\item $\m K \cong \m$ as an $R$-module.
\end{enumerate} 
\end{fact}

In this paper, we say that an $R$-submodule $K$ of $\rmQ(R)$ is {\it a fractional canonical ideal} of $R$, if $R \subseteq K \subseteq \overline{R}$ and $K \cong \rmK_R$ as an $R$-module.

One can construct many examples of almost Gorenstein rings (e.g., \cite{CCGT, GGHV, GKMT, GMP, GMTY1, GMTY2, GMTY3, GMTY4, GRTT, GTT, GTT2, H, MM, Mi, T}). The significant examples of almost Gorenstein rings are one-dimensional Cohen-Macaulay local rings of finite Cohen-Macaulay representation type and two-dimensional rational singularities. Because the origin of the theory of almost Gorenstein rings are the theory of numerical semigroup rings, there are numerous examples of almost Gorenstein numerical semigroup rings (see \cite{BF, GMP}).


\section{Proof of Theorem \ref{1.1}}

This section mainly focuses our attention on proving Theorem \ref{1.1}. 
Theorem \ref{1.1} is reduced, by induction on $\dim A$, to the case where $\dim A =1$. Let us start from the key result of dimension one. First of all, we fix our notation and assumptions.

\begin{setting}\label{4.1}
Let $(R, \m)$, $(S, \n)$ be one-dimensional Cohen-Macaulay local rings with a common residue class field $k = R/\m = S/\n$, and $f : R \to k$, $g : S \to k$ be canonical surjective maps. Then $A = R\times_k S$ is a one-dimensional Cohen-Macaulay local ring with maximal ideal $J = \m \times \n$. Since $B=R\times S$ is a module-finite birational extension over $A$, $\rmQ(A) = \rmQ(R) \times \rmQ(S)$ and $\overline{A} = \overline{R} \times \overline{S}$. 

Throughout this section, unless otherwise specified, we assume that $\rmQ(A)$ is a Gorenstein ring, $A$ admits the canonical module $\rmK_A$, and the field $k=A/J$ is infinite. Hence, all the rings $A, R$, and $S$ possess fractional canonical ideals (see \cite[Corollary 2.9]{GMP}).
\end{setting}

We denote by $K$ (resp. $L$) the fractional canonical ideal of $R$ (resp. $S$). Thus, $K$ is an $R$-submodule of $\rmQ(R)$ such that $R \subseteq K \subseteq \overline{R}$, $K \cong \rmK_R$ as an $R$-module, and $L$ is an $S$-submodule of $\rmQ(S)$ such that $S \subseteq L \subseteq \overline{S}$, $L \cong \rmK_S$ as an $S$-module.

To prove Theorem \ref{1.1} in the case where $\dim A=1$, we may assume that $A$ is not a Gorenstein ring. Hence, either $R$ or $S$ is not a DVR (see Proposition \ref{2.2}).


\subsection{The case where $R$ and $S$ are not DVRs}

In this subsection, suppose that both $R$ and $S$ are not DVRs. We then have $K \ne \overline{R}$ and $L \ne \overline{S}$, so that $K:\m \subseteq \overline{R}$, and $L:\n \subseteq \overline{S}$. Hence, because $R : \m \not\subseteq K$ and $S:\n \not\subseteq L$, we have 
$$
K:\m = K + R{\cdot}g_1, \ \ L:\n = L + S{\cdot}g_2
$$
for some $g_1 \in (R:\m) \setminus K$ and $g_2 \in (S:\n) \setminus L$. We set 
$$
X = (K\times L) + A{\cdot}\psi
$$
with $\psi = (g_1, g_2) \in  \overline{A} = \overline{R} \times \overline{S}$.  Then, $X$ is an $A$-submodule of $\overline{A}$, satisfying $A \subseteq B \subseteq X \subseteq \overline{A}$. Furthermore, we have the following which plays a key in our argument.

\begin{lem}\label{4.2}
$X \cong \rmK_A$ as an $A$-module. Hence, $X$ is a fractional canonical ideal of $A$ and $\rmr(A) = \rmr(R) + \rmr(S) + 1$.
\end{lem}

\begin{proof}
First we notice that $X:J = (K:\m) \times (L:\n)$. In fact, take an element $\alpha \in X:J$ and write $\alpha = (x, y)$, where $x \in \rmQ(R)$, $y \in \rmQ(S)$. Since $J\cdot(x, y) \subseteq X = (K\times L) + A{\cdot}\psi$, we have
$$
(m, 0)(x, y) = (mx, 0) \in X, \ \ \ (0, n)(x, y) = (0, ny) \in X
$$
for each $m \in \m$ and $n \in \n$. Hence, $(mx, 0) = (k, \ell) + (a, b)(g_1, g_2)$ for some $(k, \ell) \in K \times L$, $(a, b) \in A$. Therefore
$$
mx = k + ag_1, \ \ \ 0 = \ell + bg_2
$$
which yield that $b \in \n$ and $a \in \m$. Consequently, we have $ag_1 \in \m(R:\m) \subseteq R$, whence $x \in K:\m$. Similarly, $y \in L:\n$, as desired.

We now choose the regular elements $a \in \m$ on $R$ and $b \in \n$ on $S$. Then, $(a, b) \in A$ is a regular element on $A$. Therefore
$$
(a, b)X:_X J = (a, b)X:_{\rmQ(A)}J = (a, b)\left[X:_{\rmQ(A)}J\right]
$$
so that $[(a, b)X:_X J]/(a, b)X \cong \left[X:_{\rmQ(A)}J\right]/X$ as an $A$-module. We are now going to prove that $\left[X:_{\rmQ(A)}J\right]/X$ is cyclic as an $A$-module. Indeed, choose an element $\alpha \in X:J$ and  write $\alpha = (x, y)$, where $x \in \rmQ(R)$ and $y \in \rmQ(S)$. Then, since $X:J = (K+Rg_1) \times (L+Sg_2)$, we have
$$
x = k + cg_1 \ \ \ y = \ell + dg_2
$$
where $k \in K$, $\ell \in L$, $c \in R$, and $d \in S$. Notice that the projection $p_1 : A \to R, (a, b) \mapsto a$ is surjective. There exists $c'\in S$ such that $(c, c')\in A$ and $p_1(c, c') = c$. Similarly, let us choose $d' \in R$ such that $(d', d-c') \in A$ and $p_2(d', d-c')=d-c'$. Hence we have the equalities
\begin{eqnarray*}
\alpha = (x, y) & = & (k, \ell) + (cg_1, 0) + (0, dg_2) \\
       & = & (k, \ell) + (c, c')(g_1, g_2) + (0, (d-c')g_2) \\
       & = & (k, \ell) + (c, c')(g_1, g_2) + (d', d-c')(0, g_2)
\end{eqnarray*}
whence $\alpha \in (K \times L) + A{\cdot} \psi + A\cdot(0, g_2) = X +  A\cdot(0, g_2)$. Therefore, $\left[X:_{\rmQ(A)}J\right]/X$ is a cyclic $A$-module. Hence, $X \cong \rmK_A$ as an $A$-module.
\end{proof}

As a consequence of Lemma \ref{4.2}, we have the following, which ensures that Theorem \ref{1.1} holds when $R$ and $S$ are not DVRs, and $\dim A=1$.

\begin{cor}\label{4.3}
Suppose that $R$ and $S$ are not DVRs. Then the fiber product $A=R\times_kS$ is an almost Gorenstein ring if and only if $R$ and $S$ are almost Gorenstein rings.
\end{cor}

\begin{proof}
Notice that $JX = (\m\times\n)((K\times L) + A{\cdot}\psi) = \m(K+Rg_1) \times \n (L+Sg_2) = \m(K:\m) \times \n(L:\n) = \m K \times \n L$, where the last equality follows from the fact that $R$ and $S$ are not DVRs. Therefore, $A=R\times_kS$ is an almost Gorenstein ring if and only if $JX=J$ (Fact \ref{3.2}). The latter condition is equivalent to saying that $\m K = \m$ and $\n L = \n$, that is, $R$ and $S$ are almost Gorenstein rings, as desired.
\end{proof}

Let us note one example.

\begin{ex}\label{4.4}
Let $k$ be a field, $a, b, c, d \ge 2$ be integers. We set
$R=k[[X, Y]]/(X^a-Y^b)$, $S=k[[Z, W]]/(Z^c-W^d)$ and consider the canonical surjections $f:R \to k$, $g:S \to k$. Then
$$
R\times_kS \cong k[[X, Y, Z, W]]/{(X, Y)(Z, W) + (X^a-Y^b, Z^c-W^d)}
$$
is an almost Gorenstein local ring with $\rmr(R)=3$.
\end{ex}


For a one-dimensional case, we will show that the conditions stated in Theorem \ref{1.1} is equivalent to saying that the fiber product $A=R\times_k S$ is a generalized Gorenstein ring, which naturally covers the class of almost Gorenstein rings. Similarly, for almost Gorenstein rings, the notion is defined by a certain specific embedding of the rings into their canonical modules. Let us now recall the definition of generalized Gorenstein rings, in particular, of dimension one, which is recently proposed by the second author and S. Kumashiro (see \cite{GK} for the precise definition).

\begin{defn}[{\cite[Definition 1.2]{GK}}]\label{4.5}
Let $(R, \m)$ be a Cohen-Macaulay local ring with $\dim R=1$, possessing the fractional canonical ideal $K$, that is, $K$ is an $R$-submodule of $\overline{R}$ such that $R \subseteq K \subseteq \overline{R}$ and $K \cong \rmK_R$ as an $R$-module. We say that $R$ is {\it a generalized Gorenstein ring}, if either $R$ is Gorenstein, or $R$ is not a Gorenstein ring and $K/R$ is a free $R/\fka$-module, where $\fka = R:R[K]$. 
\end{defn}
\noindent
Notice that, by \cite[Theorem 3.11]{GMP}, $R$ is a non-Gorenstein almost Gorenstein ring if and only if $\fka = \m$, so that the above definition gives a wider class of almost Gorenstein rings. 

\medskip

We begin with the following.

\begin{prop}\label{4.6}
$\ell_A(X/A) = \ell_R(K/R) + \ell_S(L/S) + 2$.
\end{prop}

\begin{proof}
Remember that $X = (K\times L) + A{\cdot}\psi$ and $JX = J(K \times L)$. Since $JX \subseteq J \overline{A}$, we then have $J=JX \cap A$. Hence, we get the equalities
\begin{eqnarray*}
\ell_A(X/ A) & = & \ell_A(X/ JX + A) + \ell_A(JX+A/A) \\
             & = & \mu_A(X/A) + \ell_A(JX/J) \\
             & = & \rmr(R) + \rmr(S) + \ell_R(\m K/\m) + \ell_S(\n L/\n) \\
             & = & \rmr(R) + \rmr(S) + [\ell_R(K/R) + 1 -\rmr(R)] + [\ell_S(L/S) + 1 -\rmr(S)] \\
             & = & \ell_R(K/R) + \ell_S(L/S) +2
\end{eqnarray*}
as claimed.
\end{proof}


\begin{lem}\label{4.7}
Let $(R, \m)$ be a Cohen-Macaulay local ring with $\dim R=1$, possessing the fractional canonical ideal $K$. If $R$ is not a Gorenstein ring, then $\fka {\cdot}(K:\m) = \fka K$, where $\fka = R:K$.
\end{lem}

\begin{proof}
Since $R$ is not a DVR, $K:\m = K + Rg$ for some $g \in (R:\m) \setminus K$. For each $a \in \fka$ and $k \in K$, we obtain $(ag)k = (ak)g \in \m g \subseteq R$, because $ak \in \fka K \subseteq \m K \subseteq \m \overline{R}$. Hence, $ag \in R:K = \fka \subseteq \fka K$, so that $\fka (K:\m) = \fka K + \fka g = \fka K$ which completes the proof.
\end{proof}

By setting $\fka_1 = R:K$ and $\fka_2 = S:L$, we have the following.

\begin{lem}\label{4.8}
The following assertions hold true.
\begin{enumerate}[$(1)$]
\item Suppose that $R$ and $S$ are not Gorenstein rings. Then $A:X = \fka_1 \times \fka_2$.
\item Suppose that $R$ is Gorenstein, but $S$ is not a Gorenstein ring. Then $A:X = \m \times \fka_2$.
\end{enumerate}
\end{lem}

\begin{proof}
$(1)$ By Lemma \ref{4.7}, we get $\fka_1 (K:\m) = \fka_1 K$ and $\fka_2 (L:\n) = \fka_2 L$. Hence 
\begin{eqnarray*}
X\cdot(\fka_1 \times \fka_2) 
&=& (\fka_1 K + \fka_1 g) \times (\fka_2 L + \fka_2 g_2) = \fka_1 (K:\m) \times \fka_2 (L:\n) = \fka_1 K \times  \fka_2 L \\
&\subseteq&  (\m \overline{R} \cap R) \times (\n \overline{S}\cap S) = \m \times \n \subseteq A
\end{eqnarray*}
whence $\fka_1 \times \fka_2 \subseteq A:X$. Conversely, for every $(x, y) \in A:X$, we have $(K \times L)\cdot(x, y)  \subseteq A$. For each $k \in K$, we get $(k, 0)(x, y) = (kx, 0) \in A$, so that $kx \in \m$. Hence $x \in R:K=\fka_1$. Similarly, $y \in S:L=\fka_2$, as wanted.

$(2)$ Follow from the same argument as in the proof of $(1)$.
\end{proof}

As a consequence, we have the following.

\begin{cor}\label{4.9}
Suppose that $R$ and $S$ are not DVRs. Then the fiber product $A=R\times_kS$ is a generalized Gorenstein ring if and only if $R$ and $S$ are almost Gorenstein rings.
\end{cor}

\begin{proof}
The `if' part is due to Corollary \ref{4.3}. Let us make sure of the `only if' part.
Suppose that $A=R\times_kS$ is a generalized Gorenstein ring. We may assume that either $R$ or $S$ is not a Gorenstein ring. Notice that there is an isomorphism $(A/A:X)^{\oplus(r+s)} \cong X/A$ of $A$-modules, where $r=\rmr(R)$, $s = \rmr(S)$ denote the Cohen-Macaulay types of $R$ and $S$, respectively. Hence, by Proposition \ref{4.6}, 
$
(r+s) \cdot \ell_A(A/A:X) = \ell_R(K/R) + \ell_S(L/S) + 2.
$
Suppose now that both $R$ and $S$ are not Gorenstein rings. Then, because $\ell_A(A/A:X) = \ell_R(R/\fka_1) + \ell_S(S/\fka_2) - 1$, we have
\begin{eqnarray*}
(r+s) \cdot \left[\ell_R(R/\fka_1) + \ell_S(S/\fka_2) - 1\right] 
& = & \ell_R(K/R) + \ell_S(L/S) + 2 \\
& \le & (r-1)\ell_R(R/\fka_1) + (s-1) \ell_S(S/\fka_2) + 2
\end{eqnarray*}
which yield that
$
\left[\ell_S(S/\fka_2)-1\right](r+1) + \left[\ell_R(R/\fka_1) -1\right](s+1) \le 0.
$
Hence $\ell_R(R/\fka_1) = \ell_S(S/\fka_2) = 1$, so that $\fka_1 = \m$, $\fka_2 = \n$. By Fact \ref{3.2}, we conclude that $R$ and $S$ are almost Gorenstein rings. On the other hand, we consider the case where $R$ is Gorenstein, but $S$ is not a Gorenstein ring. We then have
$
(r+s)\cdot \left[\ell_S(S/\fka_2)\right] \le (s-1)\cdot\ell_S(S/\fka_2) + 2
$
which implies $\ell_S(S/\fka_2) = 1$. Hence $S$ is an almost Gorenstein ring.
\end{proof}


\subsection{The case where $R$ is a DVR and $S$ is not a DVR}

In this subsection, we assume that $R$ is a DVR and $S$ is not a DVR. 
Choose an $A$-submodule $X$ of $\rmQ(A)$ such that $A \subseteq X \subseteq \overline{A}$ and $X \cong \rmK_A$ as an $A$-module. Then, because $\rmK_B = X :B \cong R \times L$ as a $B$-module, we have $X:B = \xi\cdot (R \times L)$ for some invertible element $\xi =(\xi_1, \xi_2)\in \rmQ(A)$.  
Applying the functor $\Hom_A(-, X)$ to the exact sequence $0 \to A \overset{\iota}{\to} B \overset{\varphi}{\to} k= A/J \to 0$ where $\varphi = \left[\begin{smallmatrix}f \\ -g\end{smallmatrix}\right]$, we obtain the sequence 
$$
0 \to X:B \to X \to A/J \to 0
$$
of $A$-modules. Hence $JX \subseteq X:B \subseteq X$. 
Thus we have the following inclusions.

\begin{lem}\label{4.10}
One has
\begin{eqnarray*}
X:B ~\subseteq ~X & \subseteq& (X:B):J = (\xi_1R \times \xi_2L):J \\
&=& \xi_1(R:\m) \times \xi_2 (L:\n).
\end{eqnarray*}
In particular, $J(X:B) \subseteq JX \subseteq J\cdot[\xi_1(R:\m) \times \xi_2(L:\n)]$.
\end{lem}

We are now in a position to prove Theorem \ref{1.1} for one-dimensional case.

\begin{prop}\label{4.11}
Suppose that $R$ is a DVR and $S$ is not a DVR. Then the fiber product $A=R\times_kS$ is an almost Gorenstein ring if and only if $S$ is an almost Gorenstein ring.
\end{prop}

\begin{proof}
Suppose that $A$ is an almost Gorenstein ring. Then, since $JX=J$, we get 
$$
\n{\cdot}\xi_2L \subseteq \n \subseteq  \n{\cdot}\xi_2(L:\n) = \xi_2{\cdot}\n L
$$
because $\n (L:\n) = \n L$.
Thus $\n=\xi_2{\cdot}\n L \cong \n L$, and hence $S$ is an almost Gorenstein ring. Conversely, suppose that $S$ is an almost Gorenstein ring. We then have $\xi \in X$, because $\xi = \xi\cdot(1, 1) \in \xi\cdot(K\times L) = X:B \subseteq X$. Since $X \subseteq \overline{A} = R \times \overline{S}$, we get $\xi_1 \in R$. Therefore
\begin{eqnarray*}
JX &\subseteq& J\cdot \left[\xi_1(R:\m) \times \xi_2(L:\n) \right]\\
   &=& \xi_1\cdot\m(R:\m) \times \xi_2\cdot \n \\
   &\subseteq& \xi_1 R \times \xi_2\cdot \n
\end{eqnarray*}
where the second equality comes from the fact that $S$ is an almost Gorenstein ring, but not a DVR. Moreover, we have 
$$
JX \subseteq A\cdot(1, 0)\cdot \xi + J \xi.
$$
Indeed, for each $a \in R$, there exists $\alpha \in A$ such that $p_1 (\alpha) = a$. Hence $\alpha = (a, b)$ for some $b \in S$. Then, for every $n \in \n$, we get the equalities
\begin{eqnarray*}
(a\xi_1, \xi_2n) &=& (a, b)(\xi_1, 0) + (0, \xi_2 n) \\
                 &=& (a, b)(1, 0)\xi + (0, n) \xi
\end{eqnarray*}
which imply the required inclusion. 
Therefore, we have 
$$
JX = [JX \cap A\cdot(1, 0)\cdot \xi] + J \xi.
$$
If $JX \cap A\cdot(1, 0)\cdot \xi \subseteq J\xi$, then $JX = J\xi \cong J$. 
Let us now assume that $JX \cap A\cdot(1, 0)\cdot \xi \not\subseteq J \xi$. Take an element $\varphi \in JX \cap A\cdot(1, 0)\cdot \xi$ such that $\varphi \not\in J \xi$. Write $\varphi = (a, b)(1,0) \xi$, where $(a, b) \in A$. Then, since $\varphi \in JX \subseteq J(R \times \overline{S}) = \m \times \n \overline{S}$, we get $a\xi_1 \in \m$. Furthermore, $a\xi_1 \not\in \m \xi_1$, because $\varphi \not\in J\xi$. Hence $a \not\in \m$ and $\xi_1 \in \m$. Consequently
$$
JX \subseteq J (R \times \n) = \xi_1 R \times \xi_2 \n \subseteq JX
$$
yields that $JX = \xi(R\times\n) \cong \xi (\m\times\n) = \xi J\cong J$.
In any case, because $JX \cong J$, we conclude that $A$ is an almost Gorenstein ring.
\end{proof}

To explore the generalized Gorenstein properties of fiber products, we need more auxiliaries.

\begin{prop}\label{4.12}
The following assertions hold true.
\begin{enumerate}[$(1)$]
\item $X:J = R \times \xi_2 (S:L)$.
\item $X:B=\m \times \xi_2 L$ and $\xi_2 \rho =1$ for some $\rho \in (L:\n)\setminus L$.
\item $X=(X:B) + A$.
\item $A:X = \m \times \rho(S:L)$.
\end{enumerate}
\end{prop}

\begin{proof}
$(1), (2)$ Since $(X: B):J = X:BJ = X:J \subseteq \overline{A} = R \times \overline{S}$, we have $X:J = \xi_1(R:\m) \times \xi_2 (L:\n)$ and $\xi_1(R:\m) \subseteq R$, whence $\xi_1 \in R:(R:\m) = \m$. Because $A \subseteq X \subseteq X:J$, we get $R \subseteq \xi_1 (R:\m)$. Thus $\m = (\xi_1)$. 
Consequently, $X: B = \xi (R \times L) = \m \times \xi_2 L$, while $X:J = R \times \xi_2  (L:\n)$. Since $1 \in X:J$ and $1 \notin X:B$, there exists $\rho \in (L:\n)\setminus L$ such that $\xi_2 \rho =1$. 

$(3)$ Since $\ell_A(X/X:B) = 1$, we get $X = (X:B) + A$. 

$(4)$ Notice that $A:X = A:_A(X:B)$.  For each $(a, b) \in A:X$, $a \m \times b {\cdot}\xi_2 L \subseteq A$. We then have $b {\cdot}\xi_2 L \subseteq \n$, so that $b {\cdot}\xi_2 \subseteq \n :L = S:L$. Hence $b \in \rho (S:L)$. 
Since $\rho (S:L) \subseteq (L:\n)(S:L) = (S:L)L \subseteq \n$, we conclude that $a \in \m$. Thus $A:X \subseteq \m \times \rho (S:L) \subseteq A$.
Because $(\m \times \rho (S:L))(X:B) = \m^2 \times (S:L)L \subseteq \m^2 \times \n \subseteq A$, we have that $A:X = \m \times \rho (S:L)$.
\end{proof}

We apply Proposition \ref{4.12} to get the following corollary.

\begin{cor}\label{4.13}
$(A:X) X = \m \times (S:L) L$.
\end{cor}

\begin{proof}
The assertion follows from Proposition \ref{4.12} and $\rho (S:L) \subseteq (S:L)L$. 
\end{proof}

We need one more general lemma. 

\begin{lem}\label{lem}
Let $(R, \m)$ be a Cohen-Macaulay local ring with $\dim R=1$, possessing the fractional canonical ideal $K$. Suppose that $R$ is not a Gorenstein ring. We set $\fka = R:K$. Then the following assertions hold true. 
\begin{enumerate}[$(1)$]
\item $\fka = \fka {\cdot}(R:\m) \subseteq \fka K$.
\item $\fka = \fka K$ if and only if $K^2 = K^3$.
\end{enumerate}
\end{lem}

\begin{proof}
$(1)$ For each $a \in \fka$ and $x \in (R:\m)$, we have $(ax)k = (ak)x \in \m x \subseteq R$ for every $k \in K$, because $ak \in \fka K \subseteq \m K \subseteq \m \overline{R}$. Hence, $ax \in R:K = \fka \subseteq \fka K$, so that $\fka = \fka {\cdot}(R:\m) \subseteq \fka K$ which completes the proof.

$(2)$ If $K^2 = K^3$, then $\fka = R:K = R:R[K]$ which forms an ideal of $R[K]$. Hence $\fka K = \fka$.  Conversely, if $\fka = \fka K$, then $\fka = \fka K^n$ for every $n>0$, so that $\fka = R:R[K]$. Hence $K^2 = K^3$, because $\fka = R:K=K:K^2$ and $R:R[K] = R:K^2 = K:K^3$. 
\end{proof}

\begin{cor}
The following conditions are equivalent.
\begin{enumerate}[$(1)$]
\item $X^2 = X^3$
\item $L^2 = L^3$
\end{enumerate}
When this is the case, $\ell_A(A/(A:X)) = \ell_S(S/(S:L))$.
\end{cor}

\begin{proof}
$(1) \Rightarrow (2)$ Suppose that $X^2 = X^3$. By Lemma \ref{lem}, we get $(A:X)X = (A:X)$, so that $\m \times (S:L)L = \m \times \rho (S:L)$. Hence $(S:L)L^n =\rho^n (S:L)$ for all $n>0$. Thus $(S:L)S[L] = \rho^n (S:L) = (S:L)$ for each $n \gg 0$, because $\rho$ is an invertible element of $S[L]$. Therefore $L^2 = L^3$. 

$(2) \Rightarrow (1)$ If $L^2 = L^3$, then $S:L = S:S[L]$ which is an ideal of $S[L]$, so that $\rho (S:L) = (S:L) = (S:L)L$. Hence $X^2=X^3$, by Proposition \ref{4.12} $(4)$ and Corollary \ref{4.13}.
The last assertion follows from the fact that $\rho (S:L) = (S:L)$.
\end{proof}

Finally we reach the following.

\begin{thm}\label{4.14}
Suppose that $R$ is a DVR and $S$ is not a DVR. Then the fiber product $A=R\times_kS$ is a generalized Gorenstein ring if and only if $S$ is an almost Gorenstein ring.
\end{thm}

\begin{proof}
We only prove the `only if' part. Suppose that $A$ is a generalized Gorenstein ring. By \cite[Theorem 4.8]{GK}, we then have $A[X]/A \cong (A/A:X)^{\oplus (\ell + 1)}$ as an $A$-module, where $\ell = \rmr(S)$.  
By \cite[Theorem 4.8]{GK}, we get $X^2 = X^3$, that is $L^2 = L^3$. Hence we have the equalities 
$$
\ell_A(A[X]/A) = (\ell + 1)\ell_A(A/(A:X)) = (\ell + 1) \ell_S(S/(S:L)).
$$
\noindent
Since $X \subseteq X:J \subseteq A[X]$ (see \cite[Corollary 3.8 (1)]{GMP}), we have $X^n \subseteq (X:J)^n \subseteq A[X]$ for every $n>0$, so that $X^n = (X:J)^n =A[X]$ for every $n \gg 0$. Similarly, we have $L^n = (L:\n)^n = S[L]$ for every $n \gg 0$. Therefore
$$
A[X] = \left[R \times \xi_2(L:\n)\right]^n = R \times \xi_2^n{\cdot} S[L] = R \times \xi_2^{n+1}{\cdot} S[L] \quad \text{for all} \  n \gg 0
$$
which yields $S[L] = \xi_2{\cdot} S[L]$. Thus $\xi_2, \rho \in S[L]$ and hence $A[X]=R \times S[L]$.
Hence $\ell_A(A[X]/A) = 1 + \ell_S(S[L]/S)$.
On the other hand, by the sequence 
$$
0 \to L/S \to S[L]/S \to S[L]/L \to 0
$$
of $S$-modules, we get 
$$
\ell_S(S[L]/S) \le (\ell-1) \cdot \ell_S(S/(S:L)) + \ell_S(S/(S:L)) =\ell \cdot \ell_S(S/(S:L))
$$
because $S[L]/L \cong \rmK_{S/(S:L)}$ (see \cite[Theorem 4.8]{GK}).
Therefore
$$
(\ell + 1) \ell_S(S/(S:L)) \le 1 + \ell \cdot \ell_S(S/(S:L))
$$
which implies $\ell_S(S/(S:L)) =1$, so that $\n = S:L$. Hence $S$ is almost Gorenstein.
\end{proof}


Let us now going back to the notation as in Setting \ref{4.1}. 
By combining Corollaries \ref{4.3}, \ref{4.9}, Proposition \ref{4.11}, and Theorem \ref{4.14}, we have the following.

\begin{thm}\label{4.15}
The following conditions are equivalent. 
\begin{enumerate}[$(1)$]
\item The fiber product $A=R\times_kS$ is an almost Gorenstein ring.
\item The fiber product $A=R\times_kS$ is a generalized Gorenstein ring. 
\item $R$ and $S$ are almost Gorenstein rings.
\end{enumerate}
\end{thm}

Letting $S=R$, we get the following. 

\begin{cor}\label{4.16}
The following conditions are equivalent.
\begin{enumerate}[$(1)$]
\item The fiber product $R \times_k R$ is an almost Gorenstein ring.
\item The fiber product $R \times_k R$ is a generalized Gorenstein ring.
\item The idealization $R \ltimes \m$ is an almost Gorenstein ring.
\item $R$ is an almost Gorenstein ring.
\end{enumerate}
\end{cor}

\begin{proof}
Follow from Theorem \ref{4.15} and \cite[Theorem 6.5]{GMP}.
\end{proof}

We are now ready to prove Theorem \ref{1.1}. 
Let $(R, \m)$, $(S, \n)$ be Cohen-Macaulay local rings with $d=\dim R=\dim S>0$, $(T,\m_T)$ a regular local ring with $\dim T=d-1$, and let $f: R \to T$, $g: S \to T$ be surjective homomorphisms. Suppose that the residue class field $T/\m_T$ of $T$ is infinite. We then consider the fiber product $A=R\times_TS$, which is a Cohen-Macaulay local ring with $\dim A = d$.

With this notation, we begin with the following. 

\begin{prop}\label{4.17}
The fiber product $A=R\times_TS$ is a Gorenstein ring if and only if $R$ and $S$ are regular local rings.
\end{prop}

\begin{proof}
By Proposition \ref{2.2}, we may assume $d \ge 2$. Choose a regular system $x_1, x_2, \ldots, x_{d-1}$ of parameters for $T$. The surjectivities of $f$ and $g$ shows that $x_1, x_2, \ldots, x_{d-1}$ forms a regular sequence on $R$ and $S$. Let us write $x_i = f(a_i) = g(b_i)$ for some $a_i \in \m$ and $b_i \in \n$. We then have an isomorphism
$$
A/(\alpha_1, \alpha_2, \ldots, \alpha_{d-1}) \cong R/(a_1, a_2, \ldots, a_{d-1}) \times_{T/(x_1, x_2, \ldots, x_{d-1})} S/(b_1, b_2, \ldots, b_{d-1})
$$
of rings, which yields the required assertion, by induction arguments. 
\end{proof}

We finally reach the proof of Theorem \ref{1.1}.

\begin{proof}[Proof of Theorem \ref{1.1}]
We set $d = \dim A$. By Theorem \ref{4.15}, we may assume that $d \ge 2$ and that the assertion holds for $d-1$.  

$(1) \Rightarrow (2)$ We may assume $A$ is not a Gorenstein ring. Let us consider an exact sequence 
$$
0 \to A \to \rmK_A \to C \to 0
$$
of $A$-modules with $\mu_A(C) = \rme^0_J(C)$, where $J = (\m \times \n) \cap A$ denotes the maximal ideal of $A$. Since $A$ has an infinite residue class field, we choose a regular element $\alpha =(a, b) \in J$ on $A$ so that $\alpha$ is a superficial element for $C$ with respect to $J$, and $(f\circ p_1)(\alpha)$ forms a part of a regular system of parameters of $T$. We then have, $a \in \fkm$ and $b \in \n$ are regular elements of $R$ and $S$, respectively. Hence, we have an isomorphism $A/(\alpha) \cong R/(a) \times_{T/aT}S/(b)$ as rings. Then the hypothesis on $d$ shows that $R/(a)$ and $S/(b)$ are almost Gorenstein rings. Therefore, by \cite[Theorem 3.7]{GTT}, $R$ and $S$ are almost Gorenstein rings.

$(2) \Rightarrow (1)$ First we consider the case where $R$ and $S$ are Gorenstein rings. Choose a regular system of parameters $x_1, x_2, \ldots, x_{d-1} \in \m_T$ of $T$ and write $x_i = f(a_i) = g(b_i)$ with $a_i \in \m$ and $b_i \in \n$.  For each $1 \le i \le d-1$, we set $\alpha_i =(a_i, b_i) \in J$. Then, $\alpha_1, \alpha_2, \ldots, \alpha_{d-1}$ forms a regular sequence on $A$ and there is an isomorphism 
$$
A/(\alpha_1, \alpha_2, \ldots, \alpha_{d-1}) \cong R/(a_1, a_2, \ldots, a_{d-1}) \times_{T/(x_1, x_2, \ldots, x_{d-1})} S/(b_1, b_2, \ldots, b_{d-1})
$$
of rings. As $R/(a_1, a_2, \ldots, a_{d-1})$, $S/(b_1, b_2, \ldots, b_{d-1})$ are Gorenstein and $T/(x_1, x_2, \ldots, x_{d-1})$ is a field, we conclude that  $A/(\alpha_1, \alpha_2, \ldots, \alpha_{d-1})$ is an almost Gorenstein ring, so is $A$.
Let us assume that $R$ is Gorenstein, and $S$ is an almost Gorenstein ring, but not a Gorenstein ring. Choose an exact sequence
$$
0 \to S \to \rmK_S \to D \to 0
$$
of $S$-modules such that $\mu_S(D) = \rme^0_{\n}(D)$. Let us take an $S$-regular element $b \in \n$ such that $b$ is a superficial element for $D$ with respect to $\n$, and $g(b) \in \m_T$ is a part of a regular system of parameters of $T$. Let us choose $a \in \m$ such that $f(a) = g(b)$. Then $\alpha = (a, b) \in J$ is a regular element on $A$ and $A/(\alpha) \cong R/(a) \times_{T/aT} S/(b)$. The induction hypothesis shows that $A/(\alpha)$ is an almost Gorenstein ring, whence $A$ is almost Gorenstein. 
Finally, we are assuming that $R$ and $S$ are non-Gorenstein, almost Gorenstein rings.  Consider the exact sequences
$$
0 \to R \to \rmK_R \to C \to 0, \quad 0 \to S \to \rmK_S \to D \to 0
$$
of $R$-modules, and of $S$-modules, respectively. We then choose an $A$-regular element $\alpha =(a, b) \in J$ such that $a \in \m$ is a superficial element for $C$ with respect to $\m$, $b \in \n$ is a superficial element for $D$ with respect to $\n$, and $(f \circ p_1)(\alpha) \in \m_T$ is a part of a regular system of parameters of $T$. Then $A/(\alpha) \cong R/(a) \times_{T/aT} S/(b)$. Hence $A/(\alpha)$ is an almost Gorenstein ring, whence $A$ is almost Gorenstein, as desired.
\end{proof}


\section{Further results in dimension one}

The notion of almost Gorenstein ring in our sense originated from the works \cite{BF} of V. Barucci and R. Fr\"oberg in 1997, and \cite{GMP} of the second author, N. Matsuoka, and T. T. Phuong in 2013, where they dealt with the notion for one-dimensional Cohen-Macaulay local rings. Although the second author and S. Kumashiro have already provided a beautiful generalization of almost Gorenstein rings, say {\it generalized Gorenstein rings}, there are another directions of the generalization for one-dimensional almost Gorenstein rings, which we call {\it $2$-almost Gorenstein} and {\it nearly Gorenstein} rings.

Let us maintain the notation as in Setting \ref{4.1}. The goals of this section are stated as follows.

\begin{thm}\label{5.3}
The following conditions are equivalent.
\begin{enumerate}[$(1)$]
\item The fiber product $A=R\times_kS$ is a $2$-almost Gorenstein ring. 
\item Either $R$ is a $2$-almost Gorenstein ring and $S$ is an almost Gorenstein ring, or $R$ is an almost Gorenstein ring and $S$ is a $2$-almost Gorenstein ring.
\end{enumerate}
\end{thm}

\begin{thm}\label{5.4}
The following conditions are equivalent.
\begin{enumerate}[$(1)$]
\item The fiber product $A=R\times_kS$ is a nearly Gorenstein ring. 
\item $R$ and $S$ are nearly Gorenstein rings.
\end{enumerate}
\end{thm}

Before going ahead, we recall the definition of $2$-almost Gorenstein rings and nearly Gorenstein rings. 
For a while, let $(R, \m)$ be a Cohen-Macaulay local ring with $\dim R=1$, possessing the fractional canonical ideal $K$, so that $K$ is an $R$-submodule of $\rmQ(R)$ such that $R \subseteq K \subseteq \overline{R}$ and $K \cong \rmK_R$ as an $R$-module.

With this notation, T. D. M. Chau, the second author, S. Kumashiro, and N. Matsuoka proposed the notion of $2$-almost Gorenstein rings. 

\begin{defn}(\cite{CGKM})
We say that $R$ is {\it a $2$-almost Gorenstein ring}, if $K^2 = K^3$ and $\ell_R(K^2/K) =2$.
\end{defn}

\noindent
By \cite[Theorem 3.7]{CGKM}, the condition of $K^2 = K^3$ and $\ell_R(K^2/K)=2$ is equivalent to saying that $\ell_R(R/\fkc) = 2$, where $\fkc = R:R[K]$. The latter condition is independent of the choice of fractional canonical ideals $K$ (see \cite[Theorem 2.5]{CGKM}). Furthermore, because $R$ is an almost Gorenstein ring but not a Gorenstein ring if and only if $\ell_R(R/\fkc)=1$, that is, $\fkc =\m$ (\cite[Theorem 3.16]{GMP}), $2$-almost Gorenstein rings could be considered to be one of the successors of almost Gorenstein rings.

\medskip

Let us now state the definition of nearly Gorenstein rings, which is defined by J. Herzog, T. Hibi, and D. I. Stamate, by using the trace of canonical ideals. The notion is defined for arbitrary dimension, however, let us focus our attention on the case where $\dim R=1$, because there is no relation between almost Gorenstein and nearly Gorenstein rings for higher dimension (see \cite{HHS}). For an $R$-module $X$, let
$$
\varphi : \Hom_R(X, R) \otimes_R X \to R
$$
be the $R$-linear map defined by $\varphi(f \otimes x) = f(x)$ for every $f \in \Hom_R(X, R)$ and $x \in X$. We set 
$\tr_R(X) = \Im \varphi$ and call it {\it the trace of $X$}. The reader may consult \cite{GIK, HHS, KT, L, LP} for further information on trace ideals and modules. 

\begin{defn}(\cite{HHS})
We say that $R$ is {\it a nearly Gorenstein ring}, if $\m \subseteq \tr_R(K)$.
\end{defn}

\noindent
For every fractional ideal $X$ in $R$, we have $\tr_R(X) = (R:X){\cdot} X$ (see \cite[Lemma 2.3]{L}). 
By \cite[Proposition 6.1]{HHS}, almost Gorenstein ring is nearly Gorenstein, and the converse is not true in general, but it holds if $R$ has minimal multiplicity, that is, $v(R) = \e(R)$.


\medskip

We close this paper by proving Theorem \ref{5.3} and Theorem \ref{5.4}.

\begin{proof}[Proof of Theorem \ref{5.3}]
Thanks to Proposition \ref{2.2}, we may assume that either $R$ or $S$ is not a DVR.
Suppose that both $R$ and $S$ are not DVRs. Then $X = (K \times L) + Ag$ forms a fractional canonical ideal of $A$, where $\psi = (g_1, g_2) \in \overline{A}$, $g_1 \in (R:\m)\setminus K$, and $g_2 \in (L:\n)\setminus L$. 
By Theorem \ref{4.15}, we may assume that either $R$ or $S$ is not a Gorenstein ring. Suppose that $R$ and $S$ are not Gorenstein rings. We then have $A:X = (R:K) \times (S:L)$, so that 
$$
\ell_A(A/(A:X)) =\ell_A(B/(A:X)) -1=  \ell_R(R/(R:K)) + \ell_S(S/(S:L)) - 1.
$$ 
Hence, $\ell_A(A/(A:X)) = 2$ holds if and only if either $\ell_R(R/(R:K)) =1$ and $\ell_S(S/(S:L)) = 2$, or $\ell_R(R/(R:K)) = 2$ and $\ell_S(S/(S:L)) =1$. If $R$ is Gorenstein and $S$ is a non-Gorenstein ring. Then $A:X = \fkm \times (S:L)$ so that 
$$
\ell_A(A/(A:X)) = \ell_R(R/\fkm) + \ell_S(S/(S:L)) - 1 = \ell_S(S/(S:L))
$$
which yields that the required assertion. 
\end{proof}

\begin{proof}[Proof of Theorem \ref{5.4}]
By Proposition \ref{2.2} and Corollary \ref{4.13}, we may assume that $R$ and $S$ are not DVRs. 
If $R$ and $S$ are Gorenstein, then by Corollary \ref{4.3}, $A=R\times_kS$ is an almost Gorenstein ring, so that $A$ is nearly Gorenstein (\cite[Proposition 6.1]{HHS}). Thus, we may also assume that either $R$ or $S$ is not a Gorenstein ring. Suppose that $R$ and $S$ are not Gorenstein rings. By Lemma \ref{4.8} (1), we have that $A:X = \fka_1 \times \fka_2$, where $\fka_1 = R:K$, $\fka_2 = S:L$. Since 
$$
(A:X)X = (\fka_1 \times \fka_2)((K\times L) + A\cdot(g_1, g_2)) = \fka_1(K:\m) \times \fka_2(L:\n) = \fka_1 K \times \fka_2 L,
$$
we see that $A=R\times_kS$ is a nearly Gorenstein ring if and only if $\fka_1 K= \m$ and $\fka_2 L =\n$, in other wards, $R$ and $S$ are nearly Gorenstein. Let us now consider the case where $R$ is Gorenstein, but $S$ is not a Gorenstein ring. In this case, $A:X = \m \times \fka_2$, so that 
$$
(A:X)X = (\m \times \fka_2)((R\times L) + A\cdot(g_1, g_2)) = \m(R:\m) \times \fka_2 L.
$$
Therefore, the condition $(1)$ is equivalent to saying that $\fka_2 L=\n$, that is, $S$ is a nearly Gorenstein ring. This completes the proof.
\end{proof}



\end{document}